\documentclass[a4paper]{amsart}

\usepackage[english]{babel}
\usepackage[initials]{amsrefs}
\usepackage{mathrsfs}
\usepackage{color}
\usepackage{comment}
\usepackage{tikz-cd}
\usepackage{amssymb}

\newcommand{\bbN}{{\mathbb N}}
\newcommand{\bbQ}{{\mathbb Q}}

\newcommand{\calG}{\mathcal{G}}
\newcommand{\calS}{\mathcal{S}}

\newcommand{\calP}{\mathcal{P}}

\newcommand{\calF}{\mathcal{F}}
\newcommand{\calC}{\mathcal{C}}
\newcommand{\calR}{\mathcal{R}}
\newcommand{\calQ}{\mathcal{Q}}
\newcommand{\calT}{\mathcal{T}}
\newcommand{\calA}{\mathcal{A}}

\newcommand{\calN}{\mathcal{N}}
\newcommand{\calD}{\mathcal{D}}

\newcommand{\bs}{\backslash}

\newcommand{\id}{\operatorname{id}}
\newcommand{\im}{\operatorname{im}}

\newcommand{\isom}{\operatorname{Isom}}
\newcommand{\map}{\operatorname{Map}}

\newcommand{\homeo}{\operatorname{Homeo}}
\newcommand{\aut}{\operatorname{Aut}}

\newcommand{\Hom}{\operatorname{Hom}}

\newcommand{\sk}{\operatorname{sk}}

\DeclareMathAlphabet{\matheurm}{U}{eur}{m}{n}

\newcommand{\compactopen}{\mathscr{CO}}

\newtheorem{theorem}{Theorem}[section]
\newtheorem{lemma}[theorem]{Lemma}

\newtheorem{cor}[theorem]{Corollary}

\newtheorem{prop}[theorem]{Proposition}
\theoremstyle{definition}

\newtheorem{defn}[theorem]{Definition}

\newtheorem{remark}[theorem]{Remark}

\numberwithin{equation}{section}

\begin{document}
\title[Topological models for tree almost automorphism groups]{Topological models of finite type for tree almost automorphism groups}
\author{Roman Sauer}
\address{Karlsruhe Institute of Technology, Karlsruhe, Germany}
\email{roman.sauer@kit.edu}
\author{Werner Thumann}
\address{Max Planck Institute for Mathematics, Bonn, Germany}
\email{thumann@math.uni-bonn.de}

\begin{abstract}
We show that tree almost automorphism groups, including Neretin groups, satisfy the analogue of the $F_\infty$-finiteness condition 
in the world of totally disconnected groups: They possess a cellular action on a contractible cellular complex such that the stabilizers 
are open and compact and the restriction of the action on each $n$-skeleton is cocompact. 
\end{abstract}

\maketitle

\section{Introduction}
The Neretin groups $\calN_q$ -- introduced by Neretin as $p$-adic analogues of the diffeomorphism group of the circle -- 
are remarkable locally compact groups. They are simple, hence unimodular~\cite{kapoudjian}. But they do not possess any lattice~\cite{simple}
(the first simple such example), and they are compactly presented~\cite{boudec}. 

The Neretin groups lie in the larger family of \emph{tree almost automorphism groups} $\calA^D_{qr}$ 
(Definition~\ref{defn: tree almost automorphism groups}). It is $\calN_q\cong\calA^{\operatorname{Sym}(q)}_{q2}$. 
Compact presentability~\cite{boudec} of the groups $\calA^D_{qr}$ (denoted by $\operatorname{AAut}_D(\calT_{q,r})$ 
in~\cite{boudec}) also hold true here, and simplicity results were obtained in~\cite{caprace}. 

The Higman-Thompson groups $F_{qr}$ and $V_{qr}$ embed into $\calA^D_{qr}$ discretely and densely, respectively~(Remark \ref{rem:discreteanddense}).
For $q=2$ and $r=1$ these are the famous Thompson groups $F$ and $V$. Thompson's group $F$ was the first known torsionfree group that satisfies 
the finiteness condition of being of type $F_\infty$ and has infinite cohomological dimension~\cite{brown+geoghegan}. 

Our goal is to show that the tree almost automorphism groups $\calA^D_{qr}$ satisfy a similar $F_\infty$-property 
in the context of totally disconnected groups. 
\medskip

A discrete group~$\Gamma$ is said to be of \emph{type $F_\infty$} if it acts freely on a contractible CW-complex $X$ by cell permuting
homeomorphisms such that there are
only finitely many $\Gamma$-orbits of $n$-cells for every~$n$; in common terminology, recalled in 
Subsection~\ref{subsec: G-CW}, $X$ is a contractible 
free $\Gamma$-CW-complex of \emph{finite type}. Non-discrete totally disconnected 
groups cannot act freely on CW-complexes by cell permuting homeomorphisms. Thus we turn to an  
equivalent description of type $F_\infty$ of discrete groups which leads to the right notion in the 
totally disconnected setting. 

\begin{prop}[\cite{lueck-type}*{Lemma~4.1}]
A discrete group $\Gamma$ is of type $F_\infty$ if and only if 
there is a contractible 
proper $\Gamma$-CW-complex of finite type. 	
\end{prop}
 
Here, properness is equivalent to finiteness of stabilizers. For a totally disconnected group~$G$ we consider $G$-CW-complexes which are \emph{smooth}, 
i.e.~stabilizers are open, and \emph{proper}, i.e.~stabilizers are compact. 
A smooth $G$-CW-complex is, in particular, 
a CW-complex when discarding the group action 
(Remark~\ref{rem:gcwandcw}).  

\begin{defn}
Let $G$ be a totally disconnected group. A \emph{topological model} of $G$ is a contractible proper smooth $G$-CW-complex. 
We say that $G$ is of \emph{type $F_\infty$} if $G$ admits a topological model of finite type. 
\end{defn}

\begin{remark} 
Type $F_\infty$ implies being compactly generated and compactly presented (Proposition~\ref{prop: compact presentation}). 
\end{remark}

The theory of finiteness conditions for totally disconnected or locally compact groups is still less 
developed than the one for discrete groups. 
The notion of compact presentation was first introduced in 1964 by Kneser~\cite{kneser} which we learned from the 
forthcoming book~\cite{cornulier-delaharpe}. Abels and Tiemeyer 
introduced finiteness conditions for arbitrary locally compact groups in terms of group cohomology~\cite{at}. The relation 
of their property $C_\infty$ to the above $F_\infty$-property will be discussed in forthcoming work of the authors of this paper. 
A homological setup over the rationals for finiteness conditions of totally disconnected groups was recently developed in~\cite{weigel}. 

We are ready to state our main result which is proved in Subsection~\ref{subsec: conclusion}. It generalizes the result of Le Boudec~\cite{boudec} 
on compact presentability. 
 
\begin{theorem}\label{thm: main}
	The groups $\calA^D_{qr}$ are of type $F_\infty$. 
\end{theorem}

The ideas in Subsections \ref{sub:posetdefs}, \ref{sub:posetprops} and \ref{sub:conn_desc_link} are inspired by the finiteness results for discrete
Thompson-like groups. An interesting feature of our approach is that we simultaneously prove finiteness results for non-discrete totally disconnected 
groups and discrete groups. Indeed, in the case $D=\{1\}$, $q=2$, $r=1$ Theorem \ref{thm: main} says that
Thompson's group $V$ is of type $F_\infty$, which is a well-known result of Brown~\cite{brown-finiteness}.

\begin{remark}\label{rem:nococompact}
There is no finite-dimensional topological model for $\calA^D_{qr}$. If there were such a topological model~$X$, then the Higman-Thompson group $F_{qr}$, which is discretely 
embedded into $\calA^D_{qr}$~(Remark \ref{rem:discreteanddense}), would act properly on a contractible finite-dimensional CW-complex by cell permuting homeomorphisms. 
Since one can compute the rational homology from $X$~\cite{brown}*{Exercise~2 on p.~174}, this would imply $H_k(F_{qr},\bbQ)=0$ for all $k$ large enough.
But this would contradict the homology computations for $F_{qr}$ in~\cite{stein}*{Theorem 4.7}. 
We refer also to~\cite{weigel}*{Remark~3.11} and~\cite{caprace}*{Remark~6.16} for 
similar statements. 
\end{remark}

\begin{remark}
Other topological models for Neretin groups which are \emph{not} of finite type are 
considered in~\cite{kapoudjian-moduli}.
\end{remark}

\begin{remark} 
	Our main theorem answers a question of Castellano and Weigel~\cite{weigel}*{Question~3} in the positive where finiteness conditions are asked for Neretin groups. 
\end{remark}

One might ask whether $\calA^D_{qr}$ admits a classifying space 
of finite type with respect to the family $\compactopen$ of compact-open subgroups. 
Let us first recall this notion. 

\begin{defn}[\cite{dieck}]
	A (model of the) \emph{classifying space} of $G$ with respect to a family $\calF$ of subgroups of $G$ is a $G$-CW-complex 
	$X$ such that each isotropy group lies in $\calF$ and such that for each $H\in \calF$ the $H$-fixed point set $X^H$ is non-empty and contractible.
	We denote any such model by $E(G,\calF)$. 
\end{defn}

A reductive group $G$ over a local field admits a cocompact model of $E(G,\compactopen)$; it is given by its 
Bruhat-Tits building~\cite{lueck-survey}*{Example~4.14}. 

\begin{remark}\label{rem:noclassifying}
The group $G=\calA^D_{qr}$ does \emph{not} admit a model of $E(G,\compactopen)$ 
of finite type. The stabilizer of a point is always subconjugated to 
the stabilizer of a vertex. Hence, if there was a model with cocompact $0$-skeleton, 
then every compact-open subgroup of $G$ would be subconjugated to one of the 
finitely many representatives of conjugacy classes of stabilizers of vertices. 
By unimodularity there would be a finite upper bound on the Haar measure of 
compact-open subgroups of $G$. But this cannot be true due to Remark~\ref{rem: increasing compact-open}.  

Since a topological model of $G$ that is a CAT(0)-cell complex would be automatically a model 
of $E(G,\compactopen)$, any topological model of $G$ with the structure of a CAT(0)-cell complex cannot be cocompact nor of finite type. 
This was already observed in the paper~\cite{caprace}*{Remark~6.16} by Caprace and De Medts.
\end{remark}

Remarks \ref{rem:nococompact} and \ref{rem:noclassifying} shows that Theorem~\ref{thm: main} is the best we can expect in terms of 
finiteness properties for the groups $\calA^D_{qr}$.

\subsection{Acknowledgments}

The initial phase of this project was supported by DFG grant 1661/3-2.
Moreover, the second author wants to thank Clara L\"oh for her hospitality at the University of Regensburg 
where the first part of this project was carried out. The second author's contribution to this project in the 
final phase was carried out at the Max Planck Institute for Mathematics in Bonn.

Both authors want to thank Pierre-Emmanuel Caprace, Yves de Cornulier and Adrien Le Boudec for their helpful comments
on a draft of this paper.

\section{Notions and methods from equivariant topology}

\subsection{Equivariant CW-complexes}\label{subsec: G-CW}

We recall the notion of a $G$-CW-complex as in \cite{dieck}*{II.1} or \cite{lueck-lecturenotes}*{I.1}.   

\begin{defn}
Let $G$ be a topological group. A \emph{$G$-CW-complex} is a $G$-space $X$ endowed with a filtration 
by $G$-subspaces, called \emph{skeleta}, 
\[X^{(0)}\subset X^{(1)}\subset\ldots \subset X=\bigcup_{n\geq 0} X^{(n)}\]
such that $X$ carries the colimit topology, the $0$-skeleton $X^{(0)}$ is $\bigsqcup_{i\in I_0}G/H_i$ and for any $n\geq 1$,
the $n$-skeleton is obtained from the $(n-1)$-skeleton by 
a pushout in the category of $G$-spaces 
\begin{equation*}
	\begin{tikzcd}
		\bigsqcup_{i\in I_n}G/H_i\times S^{n-1}\arrow[hook]{d}\arrow{r} & X^{(n-1)}\arrow[hook]{d}\\
		\bigsqcup_{i\in I_n}G/H_i\times D^n\arrow{r}     & X^{(n)}
	\end{tikzcd}
\end{equation*}
The $H_i$ are closed subgroups of $G$. Conjugates of the groups $H_i$ are called 
\emph{isotropy groups} of $X$. A $G$-CW-complex is called 
\emph{proper} or \emph{smooth} if each $H_i$ is compact or open, respectively. 
It is called \emph{finite} if $\bigsqcup_nI_n$ is finite and of \emph{finite type} if $I_n$ is finite for each $n$. The latter is equivalent to the quotient 
space of every skeleton being compact. 
\end{defn}

\begin{remark}\label{rem:gcwandcw}
	A smooth $G$-CW-complex is, forgetting the group action, a CW-complex. This follows from $G/H$ being discrete whenever $H$ is
	an open subgroup. Furthermore, the action of $G$ on this CW-complex is continuous and by cell permuting homeomorphisms. An element in $G$ 
	fixing a cell setwise fixes it already pointwise. The stabilizers of cells are the isotropy groups of the $G$-CW-complex. 
	
	Vice versa, assume that $X$ is a CW-complex with an action of $G$ with the following properties:
	\begin{itemize}
		\item The action is continuous.
		\item The action is by cell permuting homeomorphisms.
		\item An element in $G$ fixing a cell setwise fixes it already pointwise.
		\item The cell stabilizers are open.
	\end{itemize}
	Then $X$ is a smooth $G$-CW-complex with isotropy groups being the cell stabilizers. Hence, if additionally the cell stabilizers
	are compact, then $X$ is a proper smooth $G$-CW-complex (cf.~\cite{dieck}*{Proposition II.1.15}). 
\end{remark}

\begin{remark}\label{rem: discretized CW}
Let $G$ be a topological group. Let $G^\delta$ be $G$ as an abstract group 
but endowed with the discrete topology. 
Every smooth $G$-CW-complex $X$ becomes 
a $G^\delta$-CW-complex since $G/H$ is canonically homeomorphic to $G^\delta/H$ for every open subgroup $H\leq G$.
Vice versa, if $X$ is a $G^\delta$-CW-complex such that the isotropy groups are open in $G$, then $X$ is also a smooth $G$-CW-complex.
\end{remark}

\begin{remark}\label{rem: remark about topology}
Because it is a standing assumption in our background reference~\cite{lueck-lecturenotes}*{Convention~1.1 on p.~6} we will assume that topological groups $G$ are 
compactly generated as spaces in Steenrod's sense~\cite{steenrod}. This is satisfied if $G$ is a locally compact group. The groups of interest to us, namely 
tree almost automorphism groups, are locally compact. If we restrict our attention to smooth $G$-CW-complexes, then the discussion in Subsections
\ref{sub:celltrading} and \ref{sub:celltradefilt} would still be valid for arbitrary topological groups because of Remark~\ref{rem: discretized CW}. 
\end{remark}

\begin{prop}\label{prop: compact presentation}
	Let $X$ be a proper smooth $G$-CW-complex. If $X$ is connected and has a finite $1$-skeleton, then $G$ is 
	compactly generated. If, in addition, $X$ is simply connected and has a finite $2$-skeleton, then 
	$G$ is compactly presented. 
\end{prop}

The above proposition is standard in the discrete setting and also not new in the locally compact setting (cf.~\cite{cornulier-delaharpe}*{Theorem~8.A.9}). 
For convenience of the reader, we point out the short argument based on the well-known lemma below. 
The lemma does not require that 
the group $G$ is countable or discrete. Hence we can apply it to 
a topological group regarded as an abstract group. 

\begin{lemma}[\cite{bridson}*{Theorem~8.10 on p.~135}]\label{lem: finite presentation from group actions}
Let $X$ be a topological space, let $G$ be an (abstract) group acting on $X$ by homeomorphisms, and let $U\subset X$ be an 
open subset such that $G\cdot U=X$. 
\begin{enumerate}
\item If $X$ is connected, then $S=\{g\in G\mid gU\cap U\ne\emptyset \}$ generates~$G$. 
\item Let $A_S$ be a set of symbols $a_s$ indexed by $S$. If $X$ and $U$ are both path-connected and $X$ is simply connected, then $G=\langle A_S\mid R\rangle$ where 
\[ R=\{a_{s_1}a_{s_2}a_{s_3}^{-1}\mid s_i\in S,~U\cap s_1U\cap s_3U\ne\emptyset,~s_1s_2=s_3 \text{ in $G$}\}.\]
\end{enumerate}
\end{lemma}

\begin{proof}[Proof of Proposition~\ref{prop: compact presentation}]
Let $X$ be a connected proper smooth $G$-CW-complex with finite $1$-skeleton~$X^{(1)}$. Since $G$ acts cocompactly on 
$X^{(1)}$ we can find an open, path-connected, relatively compact $U\subset X^{(1)}$ such that $G\cdot U=X^{(1)}$. If, in addition, $X^{(2)}$ is finite, 
we can find an open, path-connected, relatively compact $U\subset X^{(2)}$ such that $G\cdot U=X^{(2)}$. Since the relations in~$R$ are of length~$3$, 
we only have to show that the set $S$ in the statement is relatively compact. But this is clear since the fact that point stabilizers are compact 
implies that the $G$-action on $X^{(1)}$ resp.~$X^{(2)}$ is proper~\cite{lueck-lecturenotes}*{Theorem~1.23 on p.~18}. 
\end{proof}

\subsection{Equivariant cell trading}\label{sub:celltrading}

We will introduce an equivariant version of cell trading~\cite{geo}*{Proposition 4.2.1}. 
For this, we will frequently use the following fact:

\begin{remark}
	Let $X$ be a space, $Y$ a $G$-space and $H\leq G$ a closed subgroup. Then
	\begin{equation}\label{eq: adjunction}
\chi\colon \map(X,Y^H)\rightarrow \map^G(G/H\times X,Y)\hspace{8mm}f\mapsto\big[(gH,x)\mapsto g\cdot f(x)\big]
\end{equation}
	is a bijection with inverse given by $f\mapsto[x\mapsto f(eH,x)]$.
\end{remark}


Now let $(X,A)$ be a pair of $G$-CW-complexes with isotropy groups lying in the family $\calF$ which is $(\calF,n)$-connected in the following sense. 
\begin{defn}
	Let $\calF$ be a family of subgroups of $G$. We say that a pair $(X,A)$ of $G$-CW-complexes is \emph{$(\calF,n)$-connected}
	if for each $H\in\calF$ the pair $(X^H, A^H)$ is $n$-connected. Recall that a pair of spaces $(Y,B)$ is called $n$-connected if
	each path-connected component of $Y$ meets $B$ and for each $x_0\in B$ and $1\leq k\leq n$ we have $\pi_k(Y,B,x_0)=0$.
\end{defn}

Let $e$ be an equivariant $n$-cell in $X$ attached to $A$. We want to explain how we can
trade this equivariant cell with an equivariant $(n+2)$-cell of the same isotropy. More precisely, we will construct a pair $(X^t,A)$ of $G$-CW-complexes
with isotropy in $\calF$ which is $G$-homotopy equivalent to $(X,A)$ and such that $X^t$ is obtained from $X$
by deleting $e$ and adding another equivariant $(n+2)$-cell with the same isotropy. 
	
In the sequel we will freely use the adjunction~\eqref{eq: adjunction}. A 
map to $H$-fixed points will be denoted by lower case letters, and its image 
under $\chi$ by the corresponding upper case letter. 
	
The equivariant cell $e$ corresponds to a commutative 
diagram of $G$-maps
\begin{equation*}
	\begin{tikzcd}
		G/H\times S^{n-1}\arrow[hook]{d}\arrow{r}{F_0}& A\arrow[hook]{d}\\
		G/H\times D^n\arrow{r}{G_0} & X
	\end{tikzcd}
\end{equation*}
Since the pair $(X^H,A^H)$ is assumed to be $n$-connected, there is homotopy of $g_0\colon D^n\to X^H$ 
to a map into $A^H$ relative to $f_0$. 
Hence there are maps $g_1\colon D^n\to A^H$ and $g_2\colon D^{n+1}\to X^H$ such 
that $g_1\vert_{S^{n-1}}=f_0$, such that $g_2$ restricted to the upper hemisphere of $\partial D^{n+1}=S^n=D^n\cup_{S^{n-1}}D^n$ is $g_0$
and such that $g_2$ restricted to the lower hemisphere of $\partial D^{n+1}$ is $g_1$. 
By cellular approximation~\cite{bredon}*{Lemma~11.2 on p.~207} 
we may assume that $g_1$ lands in the $n$-skeleton and $g_2$ lands in the $(n+1)$-skeleton. 

For $n\ge 0$ we endow $(D^{n+2}, D^{n+1})$ with the 
following CW-structure: We have one $0$-cell lying on the equator $S^n\subset D^{n+2}$. The $n$-skeleton is the equator itself, obtained
from the $0$-cell by attaching one $n$-cell trivially. Now we attach two $(n+1)$-cells to the $n$-skeleton via the identity $S^n\rightarrow S^n$ as
gluing maps. These two $(n+1)$-cells form the upper and lower hemisphere of $\partial D^{n+2}=S^{n+1}$. We identify the second space in the pair
$(D^{n+2}, D^{n+1})$ with the lower hemisphere. Finally, we attach one $(n+2)$-cell via the identity $S^{n+1}\rightarrow S^{n+1}$ as gluing map
to obtain $D^{n+2}$.

We define $X^e$ as the \emph{elementary expansion}~\cite{lueck-lecturenotes}*{p.~61/62} of $X$ along~$G_2$, 
i.e.~$X^e$ is the following $G$-pushout: 
\begin{equation*}\label{eq: bigger elementary expansion}
	\begin{tikzcd}
		G/H\times D^{n+1}\arrow[hook]{d}\arrow{r}{G_2} & X\arrow[hook]{d}{\simeq} \\
		G/H\times D^{n+2}\arrow{r} & X^e	
	\end{tikzcd}
\end{equation*}
The space $X^e$ is obtained from $X$ by attaching an equivariant 
$(n+1)$-cell corresponding to the upper hemisphere of $S^{n+1}=\partial D^{n+2}$ and an $(n+2)$-cell which defines a $G$-CW-structure on 
$X^e$ (cf.~the remarks in~\cite{lueck-lecturenotes}*{p.~62}). The inclusion $X\subset X^e$ is a 
$G$-homotopy equivalence, even a 
strong $G$-deformation retraction, and the retraction map $p_X\colon X^e\to X$ is cellular~\cite{lueck-lecturenotes}*{p.~62}. 

Similarly, we define $A^e$ as the elementary expansion of 
$A$ along $G_1$ and the corresponding cellular retraction $p_A\colon A^e\to A$: 
\begin{equation*}\label{eq: smaller elementary expansion}
	\begin{tikzcd}
		G/H\times D^n\arrow[hook]{d}\arrow{r}{G_1} & A\arrow[hook]{d}{\simeq}\\
		G/H\times D^{n+1}\arrow{r} & A^e	
	\end{tikzcd}
\end{equation*}
Similarly as above, $A^e$ is obtained from $A$ by attaching an equivariant $n$-cell and an equivariant $(n+1)$-cell. The map $p_A$ 
pushes the $n$-cell and the $(n+1)$-cell into $A$. 
Observe the commutative diagram
\begin{equation*}
	\begin{tikzcd}
		G/H\times D^{n+1}\arrow[hook]{d}{up}\arrow[hookleftarrow]{r}{lo}& G/H\times D^n\arrow[hook]{d}{lo}\arrow{r}{G_1} & A\arrow[hook]{d}\\
		G/H\times D^{n+2}\arrow[hookleftarrow]{r}{lo}& G/H\times D^{n+1}\arrow{r}{G_2} & X 
	\end{tikzcd}
\end{equation*}
where $up$ denotes inclusion into the upper hemisphere and $lo$ denotes inclusion into the lower
hemisphere. This diagram induces a cellular inclusion $A^e\hookrightarrow X^e$ of the pushouts of the rows so that
\begin{equation}\label{eq: compatibility of expansions}
	\begin{tikzcd}
	        A\arrow[hook]{r}\arrow[hook]{d} & A^e\arrow[hook]{d}\\
	        X\arrow[hook]{r} & X^e	
	\end{tikzcd}
\end{equation}
commutes. We define $X^t$ by the $G$-pushout: 
\begin{equation}\label{eq: traded complex}
	\begin{tikzcd}
		A^e\arrow[hook]{d}\arrow{r}{p_A} & A\arrow{d}\\
		X^e\arrow{r} & X^t
	\end{tikzcd}
\end{equation}
From that pushout, $X^t$ inherits the structure of a $G$-CW-complex 
such that the right vertical map is a cellular inclusion~\cite{lueck-lecturenotes}*{1.29~on p.~21}. 
The $n$-cell $e$ gets deleted in $X^t$ and so $X^t$ has an equivariant $n$-cell less than $X$ but gains an additional 
equivariant $(n+2)$-cell in exchange. Further, $X^t$ and $X$ have the same number of equivariant $k$-cells for $n\neq k\neq n+2$.

Since $A^e\hookrightarrow X^e$ is a 
$G$-cofibration~\cite{lueck-lecturenotes}*{1.5~on p.~8} and 
$p_A$ is a $G$-homotopy equivalence, also $X^e\rightarrow X^t$ is a $G$-homotopy 
equivalence. This is a special case of~\cite{lueck-lecturenotes}*{Lemma~2.13 on p.~38} where one takes as lower row 
$X^e\leftarrow A^e\xrightarrow{p_A} A^e$ whose pushout is $X^t$ and as upper row $X^e\leftarrow A^e\xrightarrow{\id} A^e$ whose pushout is $X^e$; 
the maps from the upper to the lower row are $\id_{X^e}$, $\id_{A^e}$, and $p_A$. By putting the 
squares in~\eqref{eq: compatibility of expansions} and~\eqref{eq: traded complex} 
next to each other we obtain a bigger commutative diagram whose bottom row yields a $G$-homotopy equivalence $X\to X^t$ which is compatible with the two embeddings of~$A$, i.e.~we have a commutative triangle: 
\begin{equation*}\label{eq: relativity of cell trading}
	\begin{tikzcd}
		X\arrow{rr}{\simeq}\arrow[hookleftarrow]{dr} & & X^t\\
		& A\arrow[hook]{ur} & 
	\end{tikzcd}
\end{equation*}
We say that \emph{$X^t$ results from trading the equivariant cell $e$ in $X$ (relative to~$A$)}. 

\subsection{Cell trading on a filtration}\label{sub:celltradefilt}

Let $X$ be a $G$-CW-complex with isotropy lying in the family $\calF$. Consider a filtration $X_0\subset X_1\subset\ldots\subset X$ 
of $X$ by $G$-subcomplexes with $\bigcup_{i\geq 0} X_i=X$. We define this filtration to be \emph{$\calF$-highly connected} if
the $\calF$-connectivity of the pairs $(X_{i+1},X_i)$ tends to infinity as $i\rightarrow\infty$, i.e.
\[\forall_k\exists_n\forall_{i\geq n}\ (X_{i+1},X_i)\textnormal{ is $(\calF,k)$-connected}.\]
We say that this filtration is of \emph{finite type} if each $X_i$ is of finite type in the equivariant sense.

\begin{theorem}\label{thm:cell_trading_filtration}
	Let $X$ be a $G$-CW-complex with isotropy in $\calF$ and with a $\calF$-highly connected finite type filtration $(X_i)_{i\geq 0}$. 
	Then $X$ is $G$-homotopy equivalent to a finite type $G$-CW-complex with isotropy in $\calF$.
\end{theorem}
\begin{proof}
	First we find a sequence of numbers $n_0<n_1<n_2<\ldots$ 
	such that $(X_{i+1},X_i)$ is $(\calF,k)$-connected for each $i\ge n_k$.
	Upon passage to the sparser filtration $X_{n_1}\subset X_{n_2}\subset\ldots$ and renumbering we may assume that $n_k=k$. 
	In the following we denote the $i$-skeleton of a $G$-CW-complex $Y$ 
	by $\sk_i(Y)$.

	We inductively build the following commutative diagram of $G$-CW-complexes 
	of finite type and with isotropy in $\calF$: 
	\begin{equation}\label{eq: stair case}
	\begin{tikzcd}
		X_{0}\arrow{d}{=} \arrow[hook]{r} & X_{1}\arrow{d}{\simeq}\arrow[hook]{r} & X_{2}\arrow{d}{\simeq}\arrow[hook]{r} & X_{3}\arrow{d}{\simeq}\arrow[hook]{r}&\cdots\\
        X_{0}^{t1}\arrow[hook]{rd}\arrow[hook]{r} & X_{1}^{t1}\arrow{d}{=}\arrow[hook]{r} & X_{2}^{t1}\arrow{d}{\simeq}\arrow[hook]{r} & X_{3}^{t1}\arrow{d}{\simeq}\arrow[hook]{r}&\cdots\\
           & X_{1}^{t2}\arrow[hook]{rd} \arrow[hook]{r} &  X_{2}^{t2}\arrow{d}{=}\arrow[hook]{r} & X_{3}^{t2}\arrow{d}{\simeq}\arrow[hook]{r}&\cdots\\
           & & X_{2}^{t3}\arrow[hook]{rd}\arrow[hook]{r} &  X_{3}^{t3}\arrow{d}{=}\arrow[hook]{r} & \cdots\\
          & & & X_3^{t4}\arrow[hook]{r}\arrow[hook]{rd}&\cdots\\
		  & & & & \cdots
	\end{tikzcd}
	\end{equation}
	We enumerate the rows starting with~$0$. Further properties are: 
	\begin{enumerate}
	      \item Arrows labeled with $\simeq$ are $G$-homotopy equivalences and those with a hook indicate cellular inclusions.  
	      \item For $l\ge 1$ the $(l-1)$-skeleton of each space in the row $l$ is $\sk_{l-1}(X_{l-1}^{tl})$. 
	  	\end{enumerate}
	Next we explain how to construct the $l$-th row from the $(l-1)$-th row. For a 
	uniform treatment of the start case and the induction step, set 
	$X_i^{t0}:=X_i$, and let $l\ge 1$. 
	For $i\ge l-1$ the pair $(X_i^{t(l-1)}, X_{l-1}^{t(l-1)})$ is $(\calF, l-1)$-connected. 
	The boundary of an $(l-1)$-cell in $X_i^{t(l-1)}\bs X_{l-1}^{t(l-1)}$, $l\ge 2$, lies in 
	\[ \sk_{l-2}(X_i^{t(l-1)})=\sk_{l-2}(X_{l-2}^{t(l-1)})\subset X_{l-1}^{t(l-1)}.\]
	Hence we may and will trade every equivariant $(l-1)$-cell in $X_i^{t(l-1)}\bs X_{l-1}^{t(l-1)}$ with an equivariant $(l+1)$-cell resulting in a 
	$G$-CW-complex $X_i^{tl}$ of finite type and isotropy in $\calF$ 
	and a $G$-homotopy equivalence 
	$X_i^{t(l-1)}\to X_i^{tl}$. Doing so successively, i.e.~first trading all $(l-1)$-cells lying in $X_l^{t(l-1)}\bs X_{l-1}^{t(l-1)}$ in the  
	complexes $X_l^{t(l-1)}, X_{l+1}^{t(l-1)},\ldots$, then trading all $(l-1)$-cells lying in $X_{l+1}^{t(l-1)}\bs X_l^{t(l-1)}$ in the 
	complexes $X_{l+1}^{t(l-1)}, X_{l+2}^{t(l-1)},\ldots$ etc., results in a commutative ladder which gives the $l$-th row:
	\begin{equation*}
		\begin{tikzcd}
			X_{l-1}^{t(l-1)}\arrow[hook]{r}\arrow{d}{=} & X_{l}^{t(l-1)}\arrow{d}{\simeq}\arrow[hook]{r} & \ldots\\
			X_{l-1}^{tl}\arrow[hook]{r} & X_l^{tl}\arrow[hook]{r} & \ldots
		\end{tikzcd}
	\end{equation*}
	It is clear from the construction that the $(l-1)$-skeleton of $X_i^{tl}$ coincides with the one of $X_{l-1}^{tl}$. 
	The colimit $Y$ of the diagonal arrows in~\eqref{eq: stair case} is a $G$-CW complex of finite type whose isotropy groups are in~$\calF$. 
	Finally, we obtain a
	commutative diagram 
	\begin{equation*}\label{eq: final ladder}
	\begin{tikzcd}
		X_0\arrow[hook]{r}\arrow{d}{=} & X_1\arrow[hook]{r}\arrow{d}{\simeq}& X_2\arrow[hook]{r}\arrow{d}{\simeq}&\ldots\\
		X_0^{t1}\arrow[hook]{r}& X_1^{t2}\arrow[hook]{r} & X_2^{t3}\arrow[hook]{r}&\ldots
	\end{tikzcd}
	\end{equation*}
	where the vertical arrows are the compositions of the corresponding column arrows in~\eqref{eq: stair case}. 
	Since the horizontal arrows are $G$-cofibrations and each vertical arrow is a $G$-homotopy equivalence, 
	the induced map between the colimits of the rows, i.e.~between $X$ and $Y$, 
	is a $G$-homotopy equivalence~\cite{waner}*{Theorem~1.2}. 
\end{proof}

\section{Generalized posets and Morse theory}

In the following, we will always implicitly regard categories as topological objects via the nerve construction. More precisely,
the nerve of a category is a simplicial set and the geometric realization of this simplicial set is a CW-complex built up
from simplices glued along their faces (it is, however, not always a simplicial complex). The $0$-simplices of this complex 
are the objects of the category and the $n$-cells are in one to one correspondence with sequences of $n$ composable
non-identity arrows
\[A_0\xrightarrow{\alpha_1}A_1\xrightarrow{\alpha_2}\cdots\xrightarrow{\alpha_n}A_n\]
The top-dimensional faces of such a simplex are the simplices
\begin{gather*}
	A_1\xrightarrow{\alpha_2}A_2\xrightarrow{\alpha_3}\cdots\xrightarrow{\alpha_n}A_n\\
	A_0\xrightarrow{\alpha_1}\cdots A_{k-1}\xrightarrow{\alpha_{k+1}\circ\alpha_{k}}A_{k+1}\cdots\xrightarrow{\alpha_n} A_n\\
	A_0\xrightarrow{\alpha_1}\cdots\xrightarrow{\alpha_{n-2}}A_{n-2}\xrightarrow{\alpha_{n-1}}A_{n-1}
\end{gather*}
for $k=1,\ldots,n-1$. For the faces of the second type, it is possible that the composition $\alpha_{k+1}\circ\alpha_k$ is the identity.
In this case, the identity arrow is deleted in the sequence and we obtain a degeneracy in the gluing of the $n$-simplex.

For example, a partially ordered set (poset) can be regarded as a category $\calC$ with
\[\big|\Hom_\calC(A,B)\cup\Hom_\calC(B,A)\big|\leq 1\]
for every two objects $A,B$. In this case, the geometric realization of $\calC$ is indeed a simplicial complex.
However, we also want to work with complexes which arise from a slightly bigger class of categories:

\begin{defn}
	A \emph{generalized poset} is a small category such that $\alpha=\beta$ whenever $\alpha,\beta\colon A\rightarrow B$.
\end{defn}

Hence in a generalized poset, we allow objects to be uniquely isomorphic. Every subgroupoid of a generalized poset has trivial automorphism
groups and is therefore equivalent to a disjoint union of terminal categories. Thus, every connected component of such a subgroupoid is
contractible (recall that we implicitly regard categories as topological objects).
Since the inclusion of subcategories are inclusions of subcomplexes on the level of spaces and hence cofibrations, we can collapse connected subgroupoids in 
generalized posets and obtain homotopy equivalent generalized posets. If we collapse each connected component of the full subgroupoid
consisting of all the isomorphisms, we even get a homotopy equivalent (honest) poset which we call the \emph{underlying poset} of
the generalized poset.

More precisely, let $\calC$ be a generalized poset and $\calG$ a subgroupoid. We define the category $\calC/\calG$ as follows:
The objects of $\calC/\calG$ are equivalence classes of objects
of $\calC$ where we say that $X\sim Y$ are equivalent if there is an isomorphism $X\rightarrow Y$ in $\calG$ (which is unique).
We define
\[\Hom_{\calC/\calG}\big([X],[Y]\big):=\big\{A\rightarrow B\text{ in }\calC\ \big|\ A\in[X],~B\in[Y]\big\}\big/{\sim}\]
where elements $(A\rightarrow B)\sim(A'\rightarrow B')$ are defined to be equivalent if the diagram 
\begin{equation*}
	\begin{tikzcd}
		A\arrow{r}\arrow{d}[swap]{\calG\ni} & B\arrow{d}{\in\calG}\\
		A'\arrow{r} & B'
	\end{tikzcd}
\end{equation*}
commutes. Let $[\alpha\colon A\rightarrow B]$ and $[\beta\colon C\rightarrow D]$ be two composable arrows, i.e.~$[B]=[C]$, 
then there is a unique isomorphism $\gamma\colon B\rightarrow C$ in $\calG$ and one defines
\[[\beta\colon C\rightarrow D]\circ[\alpha\colon A\rightarrow B]:=[\beta\circ\gamma\circ\alpha\colon A\rightarrow D].\]
Set $\id_{[X]}=[\id_X]$. The next proposition is probably 
well-known and can be quickly deduced from Quillen's Theorem~A (see~\cite{op}*{Proposition 2.8} for a full proof). 

\begin{prop}\label{prop: underlying poset}
The projection $\calC\rightarrow\calC/\calG$ is a functor which induces a homotopy equivalence on the
level of spaces.
\end{prop}

We recall some more category-theoretic notions. 
The \emph{(biased) join} $\calC*\calD$ of two categories is the unique category containing $\calC$ and $\calD$ as subcategories and
exactly one extra arrow $C\rightarrow D$ for each two objects $C\in \calC$ and $D\in \calD$. Further, $\operatorname{Coone}(\calC*\calD)$ is the 
unique category containing $\calC*\calD$ as a subcategory, containing an extra object named $\mathtt{tip}$, an extra arrow $C\rightarrow\mathtt{tip}$
for each object $C\in\calC$ and an extra arrow $\mathtt{tip}\rightarrow D$ for each object $D\in\calD$. On the level of spaces, $\calC*\calD$ is homotopy equivalent to the join of (the nerves of) $\calC$ and $\calD$, and 
    $\operatorname{Coone}(\calC*\calD)$ is homotopy equivalent to the cone over 
    (the nerve of) $\calC*\calD$, thus contractible. 

If $\calC$ is a category and $X$ an object in $\calC$, then $X{\downarrow}\calC$ 
is the category whose objects are $\calC$-morphisms $X\to Y$ and whose morphisms 
are $\calC$-morphisms below $X$. Similary, one defines $\calC{\downarrow}X$ as 
the category whose objects are $\calC$-morphisms $Y\to X$ and whose morphisms are 
$\calC$-morphisms over $X$. 

\begin{remark}\label{rem:cataction}
	Let $G$ be a group acting on a small category $\calC$ by invertible functors. 
	An element $g\in G$ fixing a cell setwise already fixes its vertices
	and so fixes the cell pointwise. If $\calC$ is a generalized poset, a cell stabilizer is equal to the intersection 
	of the vertex stabilizers of that cell.
\end{remark}

\begin{remark}\label{rem:fixedpointset}
	Let $G$ act on a generalized poset $\calC$. Then denote by $\calC^G$ the full subcategory spanned by the objects fixed by $G$.
	The nerve of $\calC^G$ is the subspace of $G$-fixed points in the nerve of $\calC$. Hence the generalized poset $\calC^G$
	is the correct category representing the $G$-fixed point set in the nerve of $\calC$.
\end{remark}

Now we review the discrete Morse method for categories from \cite{op}*{Subsection 2.10} in the case of generalized posets.
Let $\calC$ be a generalized poset and $\calA\subset\calC$ a full subcategory. We are inductively adding objects to $\calA$ to build
up $\calC$. We start with $\calC_0:=\calA$. Assume we have already constructed $\calC_n$. Let $X$ be an object in 
$\calC\setminus\calC_n$ and observe the full subcategory $\calC_{n+1}$ spanned by $\calC_n$ and the object $X$.

{\it Case 1:} Assume that there is an isomorphism connecting $X$ with an element in $\calC_n$. Then the inclusion
$\calC_n\rightarrow\calC_{n+1}$ is a homotopy equivalence by \cite{op}*{Lemma 2.18}.

{\it Case 2:} Assume that there is no isomorphism connecting $X$ with an element in $\calC_n$. Then for every object $Y\in\calC_n$
there are either no arrows or only one arrow connecting $X$ with $Y$.
By \cite{op}*{Lemma 2.19} the diagram
\begin{equation}\label{eq: category pushout}
	\begin{tikzcd}
		lk_\downarrow(X)\arrow{d}\arrow{rr}&&
		\operatorname{Coone}\big(\overline{lk}_\downarrow(X)*\underline{lk}_\downarrow(X)\big)\arrow{d}\\
		\calC_n\arrow{rr}&&\calC_{n+1}
	\end{tikzcd}
\end{equation}
is a pushout on the level of spaces (in general \emph{not} a pushout in the category of small categories).
Here, $lk_\downarrow(X):=\overline{lk}_\downarrow(X)*\underline{lk}_\downarrow(X)$ is the
\emph{descending link} of $X$ with $\overline{lk}_\downarrow(X):=\calC_n{\downarrow}X$ and $\underline{lk}_\downarrow(X):=X{\downarrow}\calC_n$. 
The above pushout implies:

\begin{remark}
In~\eqref{eq: category pushout} the upper horizontal map is a cofibration since 
it is an inclusion of subcomplexes. Hence the 
Blakers-Massey homotopy excision theorem implies that 
the pair $(\calC_{n+1},\calC_n)$ is $(k+1)$-connected provided the descending link $lk_\downarrow(X)$ is $k$-connected. 
\end{remark}

If we have enough control over the connectivity of the descending links appearing in the process, we can ultimately deduce a lower connectivity
bound for the pair $(\calC,\calA)$. The obtained bound crucially depends on the order in which the objects are added. Such an order is often encoded in a
so-called Morse function: 

\begin{defn}
A \emph{generalized Morse function} on the pair $(\calC,\calA)$ is a function $f$ assigning a natural number to objects 
in $\calC\setminus\calA$ such that no two objects with the same $f$-value are connected by a non-invertible arrow. We call such $f$ \emph{well-behaved} 
if no two objects of different Morse height are joined by an isomorphism. 
\end{defn}

We then add objects by increasing
$f$-value. Note that two objects $X_1,X_2$ may have the same $f$-value. We claim that the descending links do not depend on the adding order of such objects.
This is easy to see if $X_1,X_2$ are not joined by any arrow. If there are (unique) isomorphisms between $X_1$ and $X_2$,
then the descending link $lk_\downarrow(X_1)$ (when adding $X_1$ first) is isomorphic to the descending link $lk_\downarrow(X_2)$ (when adding $X_2$
first) via the unique isomorphisms between $X_1$ and $X_2$.

So we see that when adding objects of a new Morse level, we have to check one descending link for each isomorphism component of the new
level. More precisely, we first add all the objects which are joined by isomorphisms to the lower level. This gives homotopy equivalences
by the first case above. Then we add an object which is not joined by an isomorphism to the lower level and have to check the corresponding 
descending link. We then add all the objects joined by an isomorphism to this object which again gives homotopy equivalences. We continue
with an object of another isomorphism component and repeat this process.

\begin{remark}\label{rem:desclinks}
	A well-behaved generalized Morse function $f$ on $\calC$ descends to a (honest) Morse function $f'$ (i.e.~no two objects of the same height are 
	joined by an arrow) on the underlying poset $\calC'$ of $\calC$. The underlying posets of the descending links with respect to $f$
	are the descending links with respect to $f'$ (more precisely, this is true for the descending link of each vertex of an isomorphism component 
	in $\calC$ which is added first). By Proposition~\ref{prop: underlying poset}, checking the connectivity of the descending links
	in a generalized poset is the same as checking the connectivity of the descending links in its underlying poset.
\end{remark}

\section{Tree almost automorphism groups}

\subsection{Review of tree almost automorphism groups}
For $q\geq 2$, we denote by $T_q$ the rooted $q$-regular tree, i.e.~the tree with root $o$ of degree $q$ and 
other vertices with degree $q+1$. 
We fix an embedding of $T_q$ in the plane. The embedding induces a canonical ordering, say from left to right, 
of the direct descendants of any vertex. The induced lexicographic order is a total order 
on the space of ends $B_q=\partial T_q$. We endow $B_q$ with the visual metric $d$ defined by
\[d(\xi, \eta)=\exp\big(-(\xi,\eta)_o\big)\]
where $(\xi,\eta)_o$ is the length of the common initial segment of the infinite paths 
that start at the root $o$ and represent $\xi$ and $\eta$. A ball in $B_q$ is represented by
a canonical subtree of $T_q$ containing some vertex and all of its descendants. Such a subtree is again a
rooted $q$-regular tree.

Let $\aut(T_q)$ be the group of isometries of $T_q$. We endow $\aut(T_q)$ with the compact open topology. A neighborhood basis of
$\id\in\aut(T_q)$ is given by the subsets
\[U_n=\big\{\gamma\in\aut(T_q)\ \big|\ \gamma|_{B(o,n)}=\id_{B(o,n)}\big\}\]
With this topology, $\aut(T_q)$ is a topological group which is totally disconnected and, by the Theorem of Arzel\`a-Ascoli, compact.
It is isomorphic, as a topological group, to the group of isometries of the space of ends $B_q$ endowed with the compact open topology,
i.e.~$\aut(T_q)=\isom(B_q)$ for short.

A \emph{similarity} $\phi\colon X\rightarrow Y$ of metric spaces is a homeomorphism $\phi$ such that there is a $\lambda>0$ with
$d(\phi(x_1),\phi(x_2))=\lambda d(x_1,x_2)$ for all $x_1,x_2\in X$. We call it a \emph{local similarity} if it is a homeomorphism and if for 
each $x\in X$ there is $r>0$ and $\lambda>0$ such that 
\[\phi|_{B(x,r)}\colon B(x,r)\rightarrow B(f(x),\lambda r)\]
is a $\lambda$-similarity. We call $\lambda$ the \emph{slope} of $\phi$ at $x$.
In the case of our rooted $q$-regular tree $T_q$ with space of ends $B_q$, we call a local similarity of
$B_q$ a \emph{spheromorphism}. If $\phi\colon B_q\rightarrow B_q$ is such a spheromorphism, we always find finite rooted $q$-regular
subtrees $F_1, F_2\subset T_q$ and an isometry of forests $f\colon T_q\bs F_1\rightarrow T_q\bs F_2$ which induces $\phi$. We say that $f$ 
\emph{represents} $\phi$. Consult \cite{hughes} for a detailed exposition of such correspondences. 

\begin{defn}
	We define $\calA_q$ to be the subgroup of $\homeo(B_q)$ consisting of all spheromorphisms. We endow it with the unique 
	group topology such that the inclusion $\isom(B_q)\hookrightarrow \calA_q$ is continuous and open.
\end{defn}

With this topology, $\calA_q$ is totally disconnected and locally compact. 
It is simply the topology where the sets $U_n\subset\isom(B_q)\subset\calA_q$ still form a neighborhood basis of $\id\in\calA_q$.
A neighborhood basis of $\gamma\in\calA_q$ is thus formed by the sets $\gamma U_n$. To see that multiplication is continuous, let 
$\gamma_1,\gamma_2\in\calA_q$ and $n\in\bbN$. We can
choose $n'\in\bbN$ such that $U_{n'}\gamma_2\subset \gamma_2 U_n$ (Remark \ref{rem:subnormal}). Then we have
\[\gamma_1 U_{n'}\gamma_2 U_n\subset\gamma_1\gamma_2 U_nU_n=\gamma_1\gamma_2 U_n\]
For the inverses, let $\gamma\in\calA_q$ and $n\in\bbN$. We can choose $n'\in\bbN$ such that $\gamma U_{n'}\subset U_n\gamma$
(Remark \ref{rem:subnormal}). Then we have
\[\big(\gamma U_{n'}\big)^{-1}\subset\big(U_n\gamma\big)^{-1}=\gamma^{-1}U_n^{-1}=\gamma^{-1} U_n\]

We now want to generalize these groups in two steps. For the first step, we need the following terminology:
For two metric spaces $X$ and $Y$, let $X\sqcup Y$
be the disjoint union of $X$ and $Y$ with the unique metric $d$ extending the metrics on $X$ and $Y$
and with $d(x,y)=\infty$ for $x\in X$ and $y\in Y$. By $nX$ we denote the
disjoint union $nX=X\sqcup\ldots\sqcup X$ of $n$ copies of $X$. 
Note that the summands in such a disjoint union are ordered. Thus, if in addition $X$ is ordered (for example if $X=B_q$ is
the space of ends of $T_q$ from above), we get an ordering on $nX$.

Let $r\geq 1$. Generalizing the notion of spheromorphisms, we call a local similarity $rB_q\rightarrow rB_q$ an 
\emph{$r$-spheromorphism}. Observe that the elements in $\isom(B_q)^r$ are $r$-spheromorphisms in a canonical way.
We endow $\isom(B_q)^r$ with the product topology. Then the sets $(U_n)^r\subset\isom(B_q)^r$ for varying $n$ form a 
neighborhood basis of $\id\in\isom(B_q)^r$.

\begin{defn}
	We define $\calA_{qr}$ to be the subgroup of $\homeo(rB_q)$ consisting of all $r$-spheromorphisms. We endow it with the
	unique group topology such that the inclusion $\isom(B_q)^r\hookrightarrow \calA_{qr}$ is continuous and open.
\end{defn}

With this topology, $\calA_{qr}$ is totally disconnected and locally compact. 
A neighborhood basis of $\gamma\in\calA_{qr}$ is formed by the sets $\gamma(U_n)^r$.
Obviously, we have $\calA_q=\calA_{q1}$.

For the second step, we look at certain closed subgroups of $\isom(B_q)$: Let $D\leq\operatorname{Sym}(q)$ be a subgroup of the group of permutations
on the set with $q$ elements. Observe that an element $\gamma\in\aut(T_q)$ induces
a permutation of the $q$ direct descendants of any
vertex $v\in T_q$ (after canonically identifying the full subtree below
the vertex $v$ with the full subtree below the vertex $\gamma(v)$ via the unique order-preserving tree automorphism mapping $v$ to $\gamma(v)$).
We say that $\gamma\in\aut(T_q)$ is \emph{$D$-admissible} if all these permutations are elements in $D$.
The set of all $D$-admissible elements in $\aut(T_q)=\isom(B_q)$ forms a closed subgroup of $\isom(B_q)$ which we denote
by $\isom_D(B_q)$ and which is denoted by $W(D)$ in \cite{caprace} and \cite{boudec}. Hence also $\isom_D(B_q)$ is totally disconnected and
compact. A neighborhood basis of the identity is formed by the sets $U_n^D:=\isom_D(B_q)\cap U_n$.

Let $U,V\subset B_q$ be balls and $\phi\colon U\rightarrow V$ a similarity. Since $U,V$ are represented by canonical subtrees
and $\phi$ by an isometry of these trees, it is clear what it means for $\phi$ to be \emph{$D$-admissible}. Consequently, we
can define an $r$-spheromorphism to be \emph{$D$-admissible} if it is locally determined by $D$-admissible similarities
$\phi\colon U\rightarrow V$ where $U,V$ are balls in $rB_q$.

\begin{defn}\label{defn: tree almost automorphism groups}
	We define $\calA_{qr}^D$ to be the subgroup of $\homeo(rB_q)$ consisting of all $r$-spheromorphisms which are $D$-admissible.
	We endow it with the unique group topology such that the inclusion $\isom_D(B_q)^r\hookrightarrow\calA^D_{qr}$ is continuous and open.
\end{defn}

With this topology, $\calA^D_{qr}$ is totally disconnected and locally compact. 
A neighborhood basis of $\gamma\in\calA^D_{qr}$ is formed by the sets $\gamma(U^D_n)^r$.
Obviously, we have $\calA_{qr}=\calA^{\operatorname{Sym}(q)}_{qr}$. Moreover, $\calA^{\{1\}}_{qr}$ is the
discrete Higman-Thompson group $V_{qr}$.

\begin{remark}\label{rem:discreteanddense}
	The Higman-Thompson groups $F_{qr}$ and $V_{qr}$ canonically embed into $\calA^D_{qr}$ for any $D\leq\operatorname{Sym}(q)$.
	In the first case the embedding is discrete and dense in the second. For $D=\{1\}$, this is trivial.
	For $D\neq\{1\}$, it suffices to check the following:
	\begin{itemize}
		\item $\isom_D(B_q)^r\cap F_{qr}=\{\id_{rB}\}$
		\item $\forall_{n\in\bbN}\colon U^D_n\cap V_{qr}\supsetneq\{\id_{rB}\}$
	\end{itemize}
\end{remark}

\begin{remark}\label{rem: increasing compact-open}
Let $O_n\subset \calA^D_{qr}$ be the subgroup of $D$-admissible $r$-spheromorphisms 
which are represented by an isometry of the forest $T_q\bs B(o,n) \sqcup \ldots \sqcup T_q\bs B(o, n)$. 
The subgroup $O_n$ is open and compact, and 
the increasing union $O=\bigcup O_n$ is non-compact. In particular, the 
Haar measures of the $O_n$ tend to~$\infty$. 
\end{remark}

In the following discussion we will fix $q\geq 2$, $r\geq 1$ and $D\leq\operatorname{Sym}(q)$. For better readability, we will
often drop these symbols from the notation. For example, we will write $\calA,T,B$ instead of $\calA_{qr}^D,T_q,B_q$.

\subsection{Generalized posets for tree almost automorphism groups}\label{sub:posetdefs}

We will now define a generalized poset $\calR$ for the group $\calA$ and consider two quotients $\calQ$ and $\calP$.

\begin{defn}\leavevmode
\begin{itemize}
	\item A \emph{sphero-vertex} is a local similarity $\phi\colon mB\rightarrow nB$ for $m,n\geq 1$ which is $D$-admissible,
		i.e.~which is locally determined by $D$-admissible similarities of subballs.
		We say that $m=\operatorname{lvl}(\phi)$ is the \emph{level} of $\phi$.
	\item For $m>n$, a \emph{merge map} $\mu\colon mB\rightarrow nB$ is a $D$-admissible local similarity
		such that for each summand
		$B$ in the domain $mB$ the restriction $\mu|_B\colon B\rightarrow C$ is a $D$-admissible similarity onto a ball 
		$C$ in the codomain $nB$.
	\item We say that such a merge map is \emph{very elementary} if for each summand 
		$B$ in the codomain $nB$, the preimage $\mu^{-1}(B)$ consists of at most $q$ summands in the domain $mB$
		(in other words, it is exactly one or exactly $q$ summands).
	\item A \emph{(very elementary) split map} is the inverse of a (very elementary) merge map.
	\item A \emph{transformation} is a $D$-admissible isometry $\nu\colon nB\rightarrow nB$ with $n\geq 1$. In other words,
		there is a permutation $\sigma\in\operatorname{Sym}(n)$ such that
		$\nu$ maps the $i$'th summand in the domain $nB$ onto the $\sigma(i)$'th summand in the
		codomain $nB$ via an element in $\isom_D(B)$.
	\item A \emph{strict transformation} is a $D$-admissible isometry $\nu\colon nB\rightarrow nB$ with $\sigma=\id$. In other
		words, a strict transformation is simply an element in $\isom_D(B)^n$.
\end{itemize}
\end{defn} 

\begin{remark}\label{rem:subnormal}
	An important feature of tree almost automorphism groups is a certain subnormality condition: Whenever $k\in\bbN$ and 
	$\phi\colon mB\rightarrow nB$ is a sphero-vertex, then we find $k'\in\bbN$ big enough so that $\phi\circ(U^D_{k'})^m\circ\phi^{-1}\subset (U^D_k)^n$.
	We have already used this to prove that multiplication and taking inverses in $\calA$ is continuous.
\end{remark}

The elements of the generalized poset $\calR$ are the sphero-vertices $nB\rightarrow rB$ for varying $n\geq 1$ (recall that
$r\geq 1$ is fixed).
An arrow from a sphero-vertex $\phi\colon nB\rightarrow rB$ to a sphero-vertex
$\psi\colon mB\rightarrow rB$ with $n\geq m$ is either a merge map or a transformation $\alpha$ with $\psi\circ\alpha=\phi$.
Composition is given by composing the merge maps resp.~transformations. Note that
the composition of a merge map with a transformation is again a merge map. The identities in $\calR$ are given by the identity
transformations.

Let $\calT$ be the subgroupoid of $\calR$ containing all the isomorphisms, i.e.~all the transformations. Then 
$\calP:=\calR/\calT$ is the underlying poset of $\calR$. Let $\calS$ be the subgroupoid of $\calR$ containing 
all the strict transformations. Then $\calQ:=\calR/\calS$ is a generalized poset.

The level function is a well-behaved generalized Morse function on $\calR$ and so descends to a 
generalized Morse function on $\calQ$ and to a Morse function on $\calP$. Let $\calR(n)$ resp.~$\calQ(n)$ 
resp.~$\calP(n)$ be the full subposets spanned by the objects of level at most $n$.

We define an action of $\calA$ on $\calR$ from the left by $\gamma\cdot\phi:=\gamma\circ\phi$ where $\gamma\in\calA$ and $\phi$ is
a sphero-vertex. Observe that this is a free action. It induces an action on $\calQ$ resp.~$\calP$ via $\gamma\cdot[\phi]:=[\gamma\circ\phi]$.
Since this action preserves the level of elements, the Morse filtrations $\calR(n)$, $\calQ(n)$ and $\calP(n)$ are $\calA$-invariant.

\begin{remark}\label{rem:contaction}
	Let $[\phi]$ be an object in $\calQ$ or $\calP$ with $\phi\colon nB\rightarrow rB$ a sphero-vertex. 
	By Remark \ref{rem:subnormal}, we can find $k\in\bbN$ big enough such that
	$(U^D_k)^r\circ\phi\subset\phi\circ\isom_D(B)^n$ and hence $\gamma\cdot[\phi]=[\phi]$ for each $\gamma\in (U^D_k)^r$. This shows that
	the action of $\calA$ on $\calQ$ and on $\calP$ is continuous (compare also with Remark \ref{rem:cataction}).
\end{remark}

\subsection{Properties of the generalized posets}\label{sub:posetprops}

\begin{prop}\label{prop:isotropy}
	The cell stabilizers of the action of $\calA$ on $\calQ$ are compact and open. 
\end{prop}
\begin{proof}
	Since the stabilizer of a cell is the intersection if its vertex stabilizers (Remark \ref{rem:cataction}), 
	it is sufficient to prove the statement
	for the vertex stabilizers. So let $[\phi]$ be an object in $\calQ$ and $\gamma\in\calA$ with $\gamma\cdot[\phi]=[\phi]$.
	We obtain a strict transformation $\nu\colon nB\rightarrow nB$ with $\gamma\circ\phi=\phi\circ\nu$, i.e.~$\gamma=\phi\circ\nu\circ\phi^{-1}$.
	Conversely, it is easy to see that an element of this form stabilizes $[\phi]$. Hence we have
	\[\operatorname{Stab}[\phi]=\{\phi\circ\nu\circ\phi^{-1}\ |\ \nu\textnormal{ strict transformation}\}\]
	The group of strict transformations of $nB$ is isomorphic to $\isom_D(B)^n$ which we endow with the product topology.
	The image of the injective group homomorphism
	\[\iota\colon \isom_D(B)^n\rightarrow\calA\hspace{8mm}\nu\mapsto\phi\circ\nu\circ\phi^{-1}\]
	is equal to $\operatorname{Stab}[\phi]$ and it suffices to show that $\iota$ is continuous and open.
	
	Concerning continuity, let $k\in\bbN$. By Remark \ref{rem:subnormal}, we find $k'\in\bbN$ big enough so that 
	\[(U^D_{k'})^n\circ\phi^{-1}\subset \phi^{-1}\circ(U^D_k)^r\]
	For $\gamma_1,\ldots,\gamma_n\in \isom_D(B)$ and $\alpha_1,\ldots,\alpha_n\in U^D_{k'}$ we then have
	\begin{align*}
		\iota\big(\gamma_1\alpha_1,\ldots,\gamma_n\alpha_n\big) &=\phi\big(\gamma_1\alpha_1,\ldots,\gamma_n\alpha_n\big)\phi^{-1}\\
		&= \phi\big(\gamma_1,\ldots,\gamma_n\big)\big(\alpha_1,\ldots,\alpha_n\big)\phi^{-1}\\
		&= \phi\big(\gamma_1,\ldots,\gamma_n\big)\phi^{-1}\big(\beta_1,\ldots,\beta_r\big)\\
		&= \iota\big(\gamma_1,\ldots,\gamma_n\big)\big(\beta_1,\ldots,\beta_r\big)
	\end{align*}
	for suitable $\beta_1,\ldots,\beta_r\in U^D_k$. This proves continuity.
	
	Concerning openness, let $k\in\bbN$. By Remark \ref{rem:subnormal}, we find $k'\in\bbN$ big enough so that
	\[\phi^{-1}\circ(U^D_{k'})^r\subset(U^D_k)^n\circ\phi^{-1}\]
	Similarly as above, for every $\gamma_1,\ldots,\gamma_n\in \isom_D(B)$ and $\beta_1,\ldots,\beta_r\in U^D_{k'}$, we
	find $\alpha_1,\ldots,\alpha_n\in U^D_k$ such that
	\[\iota\big(\gamma_1,\ldots,\gamma_n\big)\big(\beta_1,\ldots,\beta_r\big)=\iota\big(\gamma_1\alpha_1,\ldots,\gamma_n\alpha_n\big)\]
	This shows that $\iota$ is open.
\end{proof}

\begin{cor}\label{cor:contr}
	$\calQ$ is a proper smooth $\calA$-CW-complexes.
\end{cor}
\begin{proof}
	This follows from Remark \ref{rem:gcwandcw} together with Remark \ref{rem:contaction}, Remark \ref{rem:cataction} 
	and Proposition \ref{prop:isotropy}.
\end{proof}

\begin{prop}\label{prop:contr}
	$\calQ$ is contractible.
\end{prop}
\begin{proof}
	It suffices to prove that $\calR$ 
	is contractible (Proposition~\ref{prop: underlying poset}). It is 
	well known that a cofiltered category is contractible. Recall that a category $\calC$ is 
	\emph{cofiltered} if for every two objects $X,Y$ there is another object $Z$ and 
	morphisms from $Z$ to $X$ and $Y$ and if for every pair of morphisms 
	$f_1, f_2\colon X\to Y$ there is a morphism $g\colon Z\to X$ with 
	$f_1\circ g=f_2\circ g$. 	
	
	We now show that $\calR$ is 
	cofiltered. The second property is automatically satisfied 
	since $\calR$ is a generalized poset. 
	Let $\phi_1,\phi_2$ be two objects in $\calR$. Since $\phi_i\colon n_iB\rightarrow rB$ is a $D$-admissible local similarity, we find a partition
	$\calD_i$ of the codomain $rB$ into disjoint balls which are $D$-similar to balls in the domain $n_iB$ via $\phi_i$.
	We then find a partition $\calD$ of $rB$ into disjoint balls which both refines $\calD_1$ and $\calD_2$. Set $n:=|\calD|$, i.e.~the number
	of balls in the partition $\calD$. Let $\phi\colon nB\rightarrow rB$ be the unique order-preserving merge map mapping summands in $nB$ to
	balls in $\calD$. Then $\alpha_i:=\phi_i^{-1}\circ \phi$ are merge maps or transformations and consequently 
	yield arrows $\phi_1\leftarrow \phi\rightarrow\phi_2$.
\end{proof}

\begin{prop}\label{prop:finitetype}
	Each $\calQ(k)$ is an $\calA$-CW-complex of finite type, i.e.~there are only finitely many equivariant cells in each 
	dimension.
\end{prop}
\begin{proof}
	Let $n\in\bbN$ and $\phi_1,\phi_2$ be two sphero-vertices of level $n$. Then $\gamma:=\phi_2\circ\phi_1^{-1}$ is an element
	in $\calA$. Hence in $\calQ$ we have $\gamma\cdot[\phi_1]=[\phi_2]$. This proves that any two objects in $\calQ$ of the 
	same level are $\calA$-equivalent. So there are only finitely many $\calA$-classes of objects in $\calQ(k)$.
	In other words, the category $\calA\backslash\calQ(k)$ has only finitely many objects.

	Now let $[\phi_1],[\phi_2]$ be two objects in $\calQ$ represented by sphero-vertices $\phi_i\colon n_iB\rightarrow rB$. 
	Consider an arrow $[\phi_1]\rightarrow[\phi_2]$
	represented by a merge map or a transformation $\alpha$ with $\phi_2\circ\alpha=\phi_1$. After changing either 
	the representative $\phi_1$ or the representative $\phi_2$, 
	we may assume that $\alpha$ restricted to any summand $B$ in the domain $n_1B$ is order-preserving. 
	Observe that there are only finitely many such merge maps or transformations $n_1B\rightarrow n_2B$ with this property. This
	observation implies that $\calQ(k)$ is locally finite in the sense that any given object is the domain or codomain of only
	finitely many arrows. Hence also the category $\calA\backslash\calQ(k)$ is locally finite and thus finite because
	it has only finitely many objects. So its nerve is of finite type.
\end{proof}

\subsection{Connnectivity of the descending links}\label{sub:conn_desc_link}

In this subsection, we want to prove that the descending links in $\calP$ with respect to the Morse function
given by the level of sphero-vertices are highly connected. More precisely: For each $k\in\bbN$ there is
$n\in\bbN$ such that for all objects $X\in\calP$ of Morse height at least $n$ the descending link $lk_\downarrow(X)$
is $k$-connected. By Remark \ref{rem:desclinks}, we get the same result for the descending links of $\calR$ and $\calQ$.
We will be mainly interested in the statement for the generalized poset $\calQ$.

Fix $X=[\phi]$ an object in $\calP$ of level $n$. The objects of $lk_\downarrow(X)$ are in one to one correspondence with merge maps
$\mu\colon nB\rightarrow kB$ with $k<n$ modulo transformations on the codomain, i.e.~$\mu$ and $\mu'$ are equivalent if and only
if there is a transformation $\sigma$ with $\sigma\circ\mu=\mu'$. Equivalently, the objects are in one to one correspondence
with split maps $\nu\colon kB\rightarrow nB$ with $k<n$ modulo transformations on the domain. Denote by $[\![\nu]\!]$ the equivalence
class of a split map $\nu$ under this equivalence relation. We have $[\![\nu_1]\!]\rightarrow[\![\nu_2]\!]$ in $lk_\downarrow(X)$
if and only if the two
objects are equal or if there is a merge map $\mu$ with $\nu_2\circ\mu=\nu_1$. 
If such a merge map exists, it exists for any representatives $\nu_1,\nu_2$ and is uniquely determined by them.
Let $lk^*_\downarrow(X)$ be the subposet spanned by the elements $[\![\nu]\!]$ with $\nu$ a \emph{very elementary} split map.

\begin{prop}
	The inclusion $lk^*_\downarrow(X)\rightarrow lk_\downarrow(X)$ is a homotopy equivalence.
\end{prop}
\begin{proof}
	We want to build up $lk_\downarrow(X)$ from $lk^*_\downarrow(X)$ using a Morse function. If $\nu\colon kB\rightarrow nB$
	is a split map which is not very elementary, then we define the Morse height of 
	$[\![\nu]\!]\in lk_\downarrow(X)\setminus lk^*_\downarrow(X)$ to be the number $n-k$. It suffices to show that the descending
	link of each such vertex is contractible (and non-empty in particular).
	
	The objects of the descending link $lk_\downarrow[\![\nu]\!]$ are in one to one correspondence with merge maps $\mu$ such
	that $\nu\circ\mu$ is a split map again, modulo transformations on the domain. Denote the corresponding classes by $\langle\mu\rangle$.
	We have $\langle\mu_1\rangle\rightarrow\langle\mu_2\rangle$ in $lk_\downarrow[\![\nu]\!]$ if and only if the two objects are equal or 
	if there is a merge map 
	$\rho$ with $\mu_2\circ\rho=\mu_1$ (which exists independently of the choices of $\mu_1,\mu_2$ and is uniquely determined by them).
	
	It is now easy to see that the poset $lk_\downarrow[\![\nu]\!]$ is isomorphic to the following poset: As elements we have
	non-trivial partitions $P$ of $kB$ into disjoint subballs such that the split map $\nu$ either maps balls in $P$ homeomorphically 
	onto summands in $nB$ or splits them even further and such that at least one ball in $P$ is split non-trivially
	(i.e.~$\nu$ maps such a ball homeomorphically onto more than one summand in $nB$). We have $P_1\geq P_2$ if
	and only if $P_1$ is a refinement of $P_2$, i.e.~if and only if each ball of $P_1$ is contained in some ball of $P_2$.
	
	Now let $P_\nu$ be the partition of $kB$ which divides a summand $B$ into its $q$ maximal proper subballs if and only
	if $\nu$ splits it non-trivially. Since $\nu$ is assumed to be \emph{not} very elementary, we have $P_\nu\in lk_\downarrow[\![\nu]\!]$
	and hence $lk_\downarrow[\![\nu]\!]$ is non-empty.
	Let $P$ be an arbitrary partition in $lk_\downarrow[\![\nu]\!]$. Define $F(P)$ to be the 
	partition which divides a summand $B$ in $kB$ into its $q$ maximal proper subballs if and only if $P$ divides it into
	subballs. Then we have $P\geq F(P)$. Moreover, $F$ defines an order-preserving map (a functor)
	$F\colon lk_\downarrow[\![\nu]\!]\rightarrow lk_\downarrow[\![\nu]\!]$.  Last but not least, we have $P_\nu\geq F(P)$ for each $P$.
	
	The existence of such an $F$ implies the contractibility of the poset $lk_\downarrow[\![\nu]\!]$: The arrows 
	$P\geq F(P)$ form a natural transformation from the identity functor of $lk_\downarrow[\![\nu]\!]$ to the functor $F$.
	Furthermore, the arrows $P_\nu\geq F(P)$ form a natural transformation from the constant functor with value $P_\nu$ to $F$. 
	Natural transformations yield homotopies on the level of spaces. Hence we get a homotopy from the identity to a constant map.
	Geometrically, $F$ is a deformation retraction into the base of a cone with tip $P_\nu$.
\end{proof}

It remains to show that the connectivity of the posets $lk^*_\downarrow(X)$ tends to infinity as $\operatorname{lvl}(X)$ tends to
infinity. We first give a simpler description of this poset, starting with the case $D=\operatorname{Sym}(q)$: Recall that an element 
in this poset is a class $[\![\nu]\!]$ of very elementary split maps $\nu\colon kB\rightarrow nB$ modulo transformations on the domain. 
In such a very elementary split map, either a summand $B$ of $kB$ is not split 
at all, or else it is split into its $q$ maximal proper subballs. Thus, such a class $[\![\nu]\!]$ is already uniquely determined by the 
information which summands in $nB$ arise from a splitting of a single summand in $kB$. So the elements in the poset $lk^*_\downarrow(v)$ can
be regarded as partitions $P$ of the set ${\bf n}=\{1,\ldots,n\}$ into subsets of cardinality either $1$ or $q$ (the elements in
${\bf n}$ represent the summands in $nB$) and there has to be at least one subset of cardinality $q$. We have $P_1\geq P_2$ 
if and only if $P_1$ is a refinement of $P_2$.

Now consider the case $D=\{1\}$: The main difference to the discussion above is that whenever a summand of $kB$ is split by $\nu$ into
its $q$ maximal proper subballs, we have to remember which summands of $nB$ correspond to which of the $q$ maximal proper subballs.
In the point of view of partitions from above, we can model this information by decorating each subset of cardinality $q$ in a partition
of ${\bf n}$ representing an object in $lk^*_\downarrow(X)$ with an element in $\operatorname{Sym}(q)$. In the general
case where $D$ is an arbitrary subgroup of $\operatorname{Sym}(q)$, the decorations are elements in the set of cosets
$\operatorname{Sym}(q)/D$.

We can simplify the description even more: The geometric realization of this poset is the barycentric subdivision of the following flag complex
$\calC_n$: As vertices we have subsets $x$ of ${\bf n}$ of cardinality $q$ together with a decoration in $\operatorname{Sym}(q)/D$.
We join two such vertices 
$x_1,x_2$ by an edge if and only if they are disjoint as subsets of ${\bf n}$. The following proposition concludes the proof of the
connectivity claim.

\begin{prop}
	A lower bound for the connectivity of $\calC_n$ is given by the formula
	\[\nu(n)=\left\lfloor\frac{n-q}{2q-1}\right\rfloor-1\]
\end{prop}
\begin{proof}
	This is a special case of \cite{op}*{Theorem 4.7}: Consider a single color $\ast$ and for each element in $\operatorname{Sym}(q)/D$ an
	archetype of length $q$. Then we have 
	\[\calC_n=\mathcal{AC}_3\big(\{\ast\},\operatorname{Sym}(q)/D;{\bf n}\big)\]
	Since $m_a=q=m_r$, the formula follows.
	
	We provide the proof for the convenience of the reader.
	It is an induction over $n$. The induction start is $n\geq q$ because $q$ is the first number such that $\nu(q)\geq-1$.
	Indeed, it is obvious that $\calC_n$ is not empty in the case $n\geq q$.
	
	For the induction step, let $b$ be the vertex $\{1,\ldots,q\}\subset{\bf n}$ with arbitrary decoration. Consider the full
	subcomplex $\calC'_n$ spanned by the vertices which are disjoint from $b$. The first step is to estimate the connectivity of the
	pair $(\calC_n,\calC'_n)$ using the discrete Morse technique for simplicial complexes. Let $a$ be a vertex in $\calC_n\setminus\calC'_n$. In other
	words, $a$ intersects $b$ non-trivially. Consider the binary number $f(a)$ with $q$ digits such that the $i$'th digit
	is $0$ whenever $i\not\in a$ and $1$ whenever $i\in a$. When ordering this binary numbers in the natural way, $f$ becomes a Morse
	function building up $\calC_n$ from $\calC'_n$.
	
	We need to estimate the connectivity of the descending link $lk_\downarrow(a)$. It is the full subcomplex of $\calC_n$ spanned
	by the vertices $x$ which are disjoint from $a$ and such that $\min a<\min x$. Thus we see that it is isomorphic to the complex
	$\calC_k$ with 
	\begin{equation*}
		k = n-q-(\min a-1)\geq n-q-q+1= n-(2q-1).
	\end{equation*}
	It follows, by induction, that $lk_\downarrow(a)$ is $\nu\big(n-(2q-1)\big)$-connected. Hence, the connectivity of the pair $(\calC_n,\calC'_n)$ is
	$\nu\big(n-(2q-1)\big)+1=\nu(n)$.
	
	In the last step, we observe that the inclusion $\iota\colon\calC'_n\rightarrow\calC_n$ induces the trivial map in $\pi_m$ for
	$m\leq\nu(n)$. It then follows from the long exact homotopy sequence that $\calC_n$ is $\nu(n)$-connected. So let 
	$\varphi\colon S^m\rightarrow \calC'_n$ be a map which we can assume to be simplicial by simplicial approximation.
	Then we have $\im(\iota\circ\varphi)\subset\operatorname{star}(b)$ and $\iota\circ\varphi$ can be homotoped within the star of
	$b$ to the constant map with value the vertex $b$.
\end{proof}

\subsection{Conclusion of the proof of Theorem~\ref{thm: main}}\label{subsec: conclusion}

Consider the contractible proper smooth $\calA^D_{qr}$-CW-complex $\calQ$ (Corollary \ref{cor:contr} and Proposition \ref{prop:contr}).
It is not yet of finite type. However, we claim that the filtration $\calQ(k)$ is of finite type and 
$\compactopen$-highly connected where $\compactopen$ is the family of compact open subgroups of $\calA^D_{qr}$. We can then apply 
Theorem \ref{thm:cell_trading_filtration} and the proof is complete.
	
That each $\calQ(k)$ is of finite type follows from Proposition \ref{prop:finitetype}.
For the second statement, observe that whenever $X$ is an object in $\calQ^H$ (compare with Remark
\ref{rem:fixedpointset}) with $H\in\compactopen$ and $Y\in\calQ$ with $X\rightarrow Y$, then also $Y\in\calQ^H$: Because let $X=[\phi]$, then
for each $\gamma\in H$ the map $\phi^{-1}\circ\gamma\circ\phi$ is a strict transformation. If $Y=[\psi]$ then there is a merge map
$\mu$ with $\psi\circ\mu=\phi$. Since for each strict transformation $\nu$ such that $\mu\circ\nu$ is defined we have that
$\mu\circ\nu\circ\mu^{-1}$ is again a strict transformation, we see that $\psi^{-1}\circ\gamma\circ\psi$ is a strict transformation
for every $\gamma\in H$. In other words, $Y$ is fixed by $H$. 
	
It follows from the above observation that a descending link $lk_\downarrow^H(X)$ with respect to the filtration
$\calQ(k)^H$ of $\calQ^H$ is equal to the descending link $lk_\downarrow(X)$ with respect to the filtration $\calQ(k)$
of $\calQ$.
In Subsection \ref{sub:conn_desc_link} we have shown that the latter are highly connected. Hence also the descending links
$lk_\downarrow^H(X)$ are highly connected, uniformly in $H$. This implies that the filtration $\calQ(k)$ is $\compactopen$-highly connected 
and finishes the proof of Theorem~\ref{thm: main}.

\end{document}